\documentclass[reqno,b5paper]{amsart}

\usepackage{amsmath}
\usepackage{amssymb}
\usepackage{amsthm}
\usepackage{enumerate}
\usepackage[mathscr]{eucal}
\usepackage{eqlist}

\theoremstyle{plain}
\newtheorem{theorem}{Theorem}[section]

\newtheorem{lemma}[theorem]{Lemma}

\newtheorem{corollary}[theorem]{Corollary}
\theoremstyle{definition}
\newtheorem{definition}[theorem]{Definition}

\DeclareMathOperator{\Tr}{Tr}
\DeclareMathOperator{\Sp}{Sp}
\newtheorem{remark}{Remark}
\numberwithin{equation}{section}
\usepackage{enumerate}

\numberwithin{equation}{section}

\begin{document}
\title[Hermite-Hadamard type inequality for  operator $G$-convex]%
{Hermite-Hadamard type inequalities for operator geometrically convex functions}
\author[A. Taghavi, V. Darvish, H. M. Nazari, S. S. Dragomir]%
{A. Taghavi,  V. Darvish, H. M. Nazari, S. S. Dragomir}

\newcommand{\acr}{\newline\indent}
\address{Department of Mathematics,\\ Faculty of Mathematical
Sciences,\\ University of Mazandaran,\\ P. O. Box 47416-1468,\\
Babolsar, Iran.} \email{taghavi@umz.ac.ir,  vahid.darvish@vu.edu.au, m.nazari@stu.umz.ac.ir} 
\address{Mathematics, School of Engineering and Science Victoria University,\\ PO Box 14428, Melbourne City, MC 8001, Australia.}
\address{School of Computational and Applied Mathematics, University of the Witwatersrand, Private Bag 3, Johannesburg 2050, South Africa.} \email{sever.dragomir@vu.edu.au}


\begin{abstract}
In this paper, we introduce the concept of operator  geometrically convex functions for positive linear operators and prove some Hermite-Hadamard type inequalities for these functions. As applications, we obtain trace inequalities for operators which give some refinements of previous results.
\end{abstract}

\maketitle

\section{\textbf{Introduction and preliminaries }}\label{intro}
Let $\mathcal{A}$ be a sub-algebra of $B(H)$ stand for the commutative  $C^{*}$-algebra of  all bounded linear
operators on a complex Hilbert space $H$ with inner
product $\langle \cdot,\cdot\rangle$. An operator $A\in \mathcal{A}$ is positive and write $A\geq0$ if $\langle
Ax,x\rangle\geq0$ for all $x\in H$. Let $\mathcal{A}^{+}$ stand for all strictly positive operators in $\mathcal{A}$.

Let $A$ be a self-adjoint operator in $\mathcal{A}$. The Gelfand map establishes a $\ast$-isometrically isomorphism $\Phi$ between the set $C(\Sp(A))$ of all continuous functions defined on the spectrum of $A$, denoted $\Sp(A)$, and the $C^{*}$-algebra $C^{*}(A)$ generated by $A$ and the identity operator $1_{H}$ on $H$ as follows:

For any $f,g\in C(\Sp(A)))$ and any $\alpha, \beta\in\mathbb{C}$ we have:
\begin{itemize}
\item
$\Phi(\alpha f+\beta g)=\alpha \Phi(f)+\beta \Phi(g);$
\item
$\Phi(fg)=\Phi(f)\Phi(g)$ and $ \Phi(\bar{f})=\Phi(f)^{*};$
\item
$\|\Phi(f)\|=\|f\|:=\sup_{t\in \Sp(A)}|f(t)|;$
\item
$\Phi(f_{0})=1_{H}$ and $\Phi(f_{1})=A,$ where $f_{0}(t)=1$ and $f_{1}(t)=t$, for $t\in \Sp(A)$.
\end{itemize}
with this notation we define
$$f(A)=\Phi(f) \ \text{for all} \ \ f\in C(\Sp(A))$$
 and we call it the continuous functional calculus for a self-adjoint operator $A$.
 
 If $A$ is a self-adjoint operator and $f$ is a real valued continuous function on $\Sp(A)$, then $f(t)\geq0$ for any $t\in \Sp(A)$ implies that $f(A)\geq0$, i.e. $f(A)$ is a positive operator on $H$. Moreover, if both $f$ and $g$ are real valued functions on $\Sp(A)$ then the following important property holds:
\begin{equation}\label{e3}
f(t)\geq g(t) \ \ \text{for any} \ \ t\in \Sp(A)\ \ \text{implies that} \ \ f(A)\geq g(A),
\end{equation}
in the operator order of $B(H)$, see \cite{zhu}.\\

Let $I$ be an interval in $\mathbb{R}$. Then $f:I\to\mathbb{R}$ is said to be convex function if 
$$f(\lambda a+(1-\lambda)b)\leq \lambda f(a)+(1-\lambda) f(b)$$
for $a,b\in I$ and $\lambda\in [0,1]$.

The following inequality holds for any convex function $f$ defined on $\mathbb{R}$
\begin{equation}\label{ro}
(b-a)f\left(\frac{a+b}{2}\right)\leq \int_{a}^{b} f(x)dx\leq (b-a)\frac{f(a)+f(b)}{2}, \ \ a,b\in\mathbb{R}.
\end{equation}
It was firstly discovered by Hermite in 1881 in the journal Mathesis (see \cite{mit}). But this result was nowhere mentioned in the mathematical literature and was not widely known as Hermite's result \cite{pec}.

Beckenbach, a leading expert on the history and the theory of convex functions, wrote that this inequality was proven by
Hadamard in 1893 \cite{bec}. In 1974, Mitrinovi\v{c} found Hermite’s note in Mathesis \cite{mit}. Since (\ref{ro}) was known as Hadamard’s
inequality, the inequality is now commonly referred as the Hermite-Hadamard inequality \cite{pec}.

\begin{definition}\cite{nic}
A continuous function $f:I\subset \mathbb{R}^{+}\to\mathbb{R}^{+}$ is said to be geometrically convex function (or multiplicatively  convex function) if 
$$f(a^{\lambda}b^{1-\lambda})\leq f(a)^{\lambda}f(b)^{1-\lambda}$$
for $a,b\in I$ and $\lambda\in[0,1]$.
\end{definition}
The author of \cite{sad} established  the Hermite-Hadamard type inequalities for geometrically convex functions as follows:
\begin{theorem}
Let $f:I\subseteq \mathbb{R}^{+}\to \mathbb{R}^{+}$ be a geometrically convex function and $a,b\in I$ with $a< b$. If $f\in L^{1}[a,b]$, then
\begin{eqnarray*}
f(\sqrt{ab})&\leq& \frac{1}{\ln b-\ln a}\int_{a}^{b}\frac{1}{t}\sqrt{f(t)f\left(\frac{ab}{t}\right)}dt\\
&\leq& \frac{1}{\ln b-\ln a}\int_{a}^{b}\frac{f(t)}{t}dt\\
&\leq& \frac{f(b)-f(a)}{\ln f(b)-\ln f(a)}\\
&\leq&\frac{f(a)+f(b)}{2}.
\end{eqnarray*}
\end{theorem}
\noindent By changing variables $t=a^{\lambda}b^{1-\lambda}$ we have $$\frac{1}{\ln b-\ln a}\int_{a}^{b}\frac{f(t)}{t}dt=\int_{0}^{1}f(a^{\lambda}b^{1-\lambda})d\lambda.$$
\begin{remark}
It is well-known that for positive numbers $a$ and $b$
$$\min\{a,b\}\leq G(a,b)=\sqrt{ab}\leq L(a,b)=\frac{b-a}{\ln b-\ln a}\leq A(a,b)=\frac{a+b}{2}\leq \max\{a,b\}.$$
\end{remark}
The author of \cite{sad2} mentioned the following inequality, but here we provide a short proof which gives a refinement for above theorem.
\begin{theorem}
Let $f$ be a geometrically convex function defined on $I$ a sub-interval of $\mathbb{R}^{+}$. Then, we have
$$f(\sqrt{ab})\leq \frac{1}{\ln b-\ln a}\int_{a}^{b}\frac{1}{t}\sqrt{f(t)f\left(\frac{ab}{t}\right)}dt\leq \sqrt{f(a)f(b)}$$
for $a,b\in I$.
\end{theorem}
\begin{proof}
Since $f$ is geometrically convex function,  we can write
\begin{eqnarray*}
f(\sqrt{ab})&=&f\left(\sqrt{(a^{\lambda}b^{1-\lambda})(a^{1-\lambda}b^{\lambda})}\right)\\
&\leq& \sqrt{f(a^{\lambda}b^{1-\lambda})f(a^{1-\lambda}b^{\lambda})}\\
&\leq& \sqrt{f(a)^{\lambda}f(b)^{1-\lambda}f(a)^{1-\lambda}f(b)^{\lambda}}\\
&=&\sqrt{f(a)f(b)}.
\end{eqnarray*}
for all $\lambda\in[0,1]$.\\
So, we have
\begin{equation}\label{e11}
f(\sqrt{ab})\leq \sqrt{f(a^{\lambda}b^{1-\lambda})f(a^{1-\lambda}b^{\lambda})}\leq \sqrt{f(a)f(b)}.
\end{equation}
Integrate (\ref{e11}) over $[0,1]$, we have
$$f(\sqrt{ab})\leq \frac{1}{\ln b-\ln a}\int_{a}^{b}\frac{1}{t}\sqrt{f(t)f\left(\frac{ab}{t}\right)}dt\leq \sqrt{f(a)f(b)}.$$
\end{proof}

\begin{lemma}\label{htn}\cite[Page. 156]{nic}
Suppose that $I$ is a subinterval of $\mathbb{R}^{+}$ and $f:I\to (0,\infty)$ is a geometrically convex function. Then 
$$F=\log\circ f\circ \exp : \log (I) \to \mathbb{R}$$ is a convex function. Conversely, if $J$ is an interval  for which $\exp (J)$ is a subinterval of $\mathbb{R}^{+}$ and $F: J\to \mathbb{R}$ is a convex function, then $$f=\exp\circ F\circ \log: \exp (J) \to \mathbb{R}^{+}$$ is geometrically convex function.
\end{lemma}
\begin{theorem}
Let $f$ be a geometrically convex function defined on $[a,b]$ such that $0<a<b$. Then, we have
\begin{eqnarray*}
f(\sqrt{ab})&\leq& \sqrt{\left(f(a^{\frac{3}{4}}b^{\frac{1}{4}})f(a^{\frac{1}{4}}b^{\frac{3}{4}})\right)}\\
&\leq& \exp\left(\frac{1}{\log b-\log a}\int_{a}^{b}\frac{\log f(t)}{t}dt\right)\\
&\leq&\sqrt{f(\sqrt{ab})}.\sqrt[4]{f(a)}.\sqrt[4]{f(b)}\\
&\leq& \sqrt{f(a)f(b)}
\end{eqnarray*}
for $a, b\in I$.
\end{theorem}
\begin{proof}
Let $f:[a,b]\to \mathbb{R}$ be a geometrically convex function. So, by Lemma \ref{htn} we have 
$$F(x)=\log\circ f\circ \exp(x): [\log a, \log b]\to \mathbb{R}$$
is convex.\\
Then, by \cite[Remark 1.9.3]{nic2}
\begin{eqnarray*}
F\left(\frac{\log a+\log b}{2}\right)&\leq&\frac{1}{2}\left(F\left(\frac{3\log a+\log b}{4}\right)+F\left(\frac{\log a+3\log b}{4}\right)\right)\\
&\leq&\frac{1}{\log b-\log a}\int_{\log a}^{\log b} F(x)dx\\
&\leq&\frac{1}{2}\left(F\left(\frac{\log a+\log b}{2}\right)+\frac{F(\log a)+F(\log b)}{2}\right)\\
&\leq&\frac{F(\log a)+F(\log b)}{2}.
\end{eqnarray*}
By definition of $F$, we obtain
\begin{eqnarray*}
\log\circ f\circ \exp(\log\sqrt{ab})&\leq&\frac{1}{2}\left(\log\circ f\circ \exp\left(\log a^{\frac{3}{4}}b^{\frac{1}{4}}\right)+\log\circ f\circ \exp\left(\log a^{\frac{3}{4}}b^{\frac{1}{4}}\right)\right)\\
&\leq&\frac{1}{\log b-\log a}\int_{\log a}^{\log b} \log\circ f\circ \exp(x)dx\\
&\leq&\frac{1}{2}\left(\log\circ f\circ \exp\left(\log a^{\frac{1}{2}}b^{\frac{1}{2}}\right)+\frac{\log\circ f\circ \exp(\log a)+\log\circ f\circ \exp(\log b)}{2}\right)\\
&\leq&\frac{\log\circ f\circ \exp(\log a)+\log\circ f\circ \exp(\log b)}{2}.
\end{eqnarray*}
It follows that
\begin{eqnarray*}
\log f(\sqrt{ab})&\leq&\frac{1}{2}\left(\log f( a^{\frac{3}{4}}b^{\frac{1}{4}})+\log f (a^{\frac{3}{4}}b^{\frac{1}{4}})\right)\\
&\leq&\frac{1}{\log b-\log a}\int_{\log a}^{\log b} \log\circ f\circ \exp(x)dx\\
&\leq&\frac{1}{2}\left(\log  f \left(a^{\frac{1}{2}}b^{\frac{1}{2}}\right)+\frac{\log f (a)+\log f (b)}{2}\right)\\
&\leq&\frac{\log f (a)+\log f (b)}{2}.
\end{eqnarray*}
Since $\exp(x)$ is increasing, we have 
\begin{eqnarray*}
f(\sqrt{ab})&\leq& \sqrt{\left(f(a^{\frac{3}{4}}b^{\frac{1}{4}})f(a^{\frac{1}{4}}b^{\frac{3}{4}})\right)}\\
&\leq& \exp\left(\frac{1}{\log b-\log a}\int_{\log a}^{\log b}\log f(\exp (x))dx\right)\\
&\leq&\sqrt{f(\sqrt{ab})}.\sqrt[4]{f(a)}.\sqrt[4]{f(b)}\\
&\leq& \sqrt{f(a)f(b)}.
\end{eqnarray*}
Using change of variable $t=\exp(x)$ to obtain the desired result.
\end{proof}

The author of  \cite[p. 158]{nic} showed that every polynomial $P(x)$ with non-negative coefficients is a geometrically convex function on $[0,\infty)$. More generally, every real analytic function $f(x)=\sum_{n=0}^{\infty}c_{n}x^{n}$ with non-negative coefficients is geometrically convex function on $(0,R)$ where $R$ denotes the radius of convergence.  This gives some different examples of geometrically convex function. It is easy to show that $\exp(x)$ is geometrically convex function.

In this paper, we introduce the concept of operator geometrically convex functions and prove the Hermite-Hadamard type inequalities for these class of functions. These results lead us to obtain some inequalities for trace functional of operators.    
 
\section{\textbf{Inequalities for operator geometrically convex functions}}\label{gg}

In this section, we prove Hermite-Hadamard type inequality for operator geometrically convex function. 

In \cite{dra3} Dragomir investigated the operator version of the Hermite-Hadamard inequality for
operator convex functions. Let $f:I\to\mathbb{R}$ be an operator convex function on the interval $I$ then, for any self-adjoint operators $A$ and $B$ with spectra in $I$, the following inequalities holds
\begin{eqnarray}
f\left(\frac{A+B}{2}\right)&\leq&2\int_{\frac{1}{4}}^{\frac{3}{4}} f(tA+(1-t)B)dt\\
&\leq& \frac{1}{2}\left[f\left(\frac{3A+B}{4}\right)+f\left(\frac{A+3B}{4}\right)\right]\\
&\leq& \int_{0}^{1}f\left((1-t)A+tB\right)dt\nonumber\\
&\leq&\frac{1}{2}\left[f\left(\frac{A+B}{2}\right)+\frac{f(A)+f(B)}{2}\right]\\
&\leq& \frac{f(A)+f(B)}{2}, \label{kdv}
\end{eqnarray} 
for the first inequality in above, see \cite{taghavi}.

To give operator geometrically convex function definition, we need following lemmas. 
\begin{lemma}\cite[Lemma 3]{nag}\label{com}
Let $A$ and $B$ be two operators in $\mathcal{A}^{+}$, and $f$ a continuous function on $\Sp(A)$. Then, $AB=BA$ implies that $f(A)B=Bf(A)$.
\end{lemma}
\noindent Since $f(t)=t^{\lambda}$ is continuous function for $\lambda\in [0,1]$  and $\mathcal{A}$ is a commutative $C^{*}$-algebra, we have
$A^{\lambda}B=BA^{\lambda}$. Moreover, by applying above lemma for $f(t)=t^{1-\lambda}$ again, we have
$A^{\lambda}B^{1-\lambda}=B^{1-\lambda}A^{\lambda}$, for  operators $A$ and $B$ in $\mathcal{A}^{+}$. It means $A^{\lambda}$ and $B^{1-\lambda}$ commute together whenever $A$ and $B$ commute.
\begin{lemma}\label{gconvex}
Let $A$ and $B$ be two operators in $\mathcal{A}^{+}$. Then $$\{A^{\lambda}B^{1-\lambda} : 0\leq \lambda\leq 1\}$$ is convex.
\end{lemma}
\begin{proof}
We know that $\{\lambda A+(1-\lambda)B : 0\leq \lambda \leq1\}$ is convex for arbitrary operator $A$ and $B$. So, $\{\lambda \log A+(1-\lambda)\log B : 0\leq \lambda \leq 1\}$ is convex. Since $A$ and $B$ are commutative and knowing that $e^{f}$ is convex when $f$ is convex, we have
\begin{eqnarray*}
e^{\left(\lambda \log A+(1-\lambda) \log B\right)}&=&e^{\lambda \log A}e^{(1-\lambda)\log B}\\
&=&A^{\lambda}B^{1-\lambda}.
\end{eqnarray*}
So, $A^{\lambda}B^{1-\lambda}$ is convex for $0\leq\lambda\leq1$.
\end{proof}

\begin{lemma}\cite[Theorem 5.3]{zhu}
Let $A$ and $B$ be in a Banach algebra such that $AB=BA$. Then 
$$\Sp(AB)\subset\Sp (A)\Sp (B).$$
 \end{lemma}
 Let $A$ and $B$ be two positive operators in $\mathcal{A}$ with spectra in $I$. Now, Lemma \ref{com} and functional calculus \cite[Theorem 10.3 (c)]{zhu} imply that
 \begin{equation*}
 \Sp(A^{\lambda}B^{1-\lambda})\subset \Sp(A^{\lambda})\Sp(B^{1-\lambda})=\Sp(A)^{\lambda}\Sp(B)^{1-\lambda}\subseteq I
 \end{equation*}
 for $0\leq\lambda\leq1$.
\begin{definition}
A continuous function $f:I\subseteq \mathbb{R}^{+}\to\mathbb{R}^{+}$ is said to be  operator geometrically convex  if 
$$f(A^{\lambda}B^{1-\lambda})\leq f(A)^{\lambda}f(B)^{1-\lambda}$$
for $A, B\in \mathcal{A}^{+}$ such that $\Sp (A)$, $\Sp(B)\subseteq I$.
\end{definition}
Now, we are ready to prove Hermite-Hadamard type inequality for operator geometrically convex functions.
\begin{theorem}\label{thm}
Let $f$ be an operator geometrically convex function. Then, we have
\begin{equation}\label{mr}
\log f(\sqrt{AB})\leq \int_{0}^{1}\log f(A^{t}B^{1-t})dt\leq \log \sqrt{f(A)f(B)}
\end{equation}
for $0\leq t\leq 1$ and $A,B\in \mathcal{A}^{+}$ such that $\Sp (A), \Sp(B)\subseteq I$.
\end{theorem}
\begin{proof}
Since $f$ is operator geometrically convex function, we have
$f(\sqrt{AB})\leq \sqrt{f(A)f(B)}$. Let replace $A$ and $B$ by $A^{t}B^{1-t}$ and $A^{1-t}B^{t}$ respectively, we obtain
\begin{equation}
f(\sqrt{AB})\leq\sqrt{f(A^{t}B^{1-t})f(A^{1-t}B^{t})}.
\end{equation}
It is well-known that $\log t$ is operator monotone function on $(0,\infty)$ (see \cite{zhan}), i.e., $\log t$ is operator monotone function if $\log A\leq \log B$ when $A\leq B$. So, by above inequality, we have 
\begin{eqnarray*}
\log f(\sqrt{AB})&\leq& \log \sqrt{f(A^{t}B^{1-t})f(A^{1-t}B^{t})}\\
&=&\frac{1}{2}\log \left(f(A^{t}B^{1-t})f(A^{1-t}B^{t})\right)\\
&=&\frac{1}{2}\left(\log f(A^{t}B^{1-t})+\log f(A^{1-t}B^{t})\right).
\end{eqnarray*}
Therefore, 
$$\log f(\sqrt{AB})\leq \frac{1}{2}\left(\log f(A^{t}B^{1-t})+\log f(A^{1-t}B^{t})\right).$$
Integrate above inequality over $[0,1]$, we can write the following
\begin{eqnarray}
\int_{0}^{1}\log f(\sqrt{AB})dt&\leq& \frac{1}{2}\left(\int_{0}^{1}\log f(A^{t}B^{1-t})dt+\int_{0}^{1}\log f(A^{1-t}B^{t})dt\right)\nonumber\\
&=&\int_{0}^{1}\log f(A^{t}B^{1-t})dt.\label{jv1}
\end{eqnarray}
The last above equality  follows by knowing that 
$$\int_{0}^{1}\log f(A^{t}B^{1-t})dt=\int_{0}^{1}\log f(A^{1-t}B^{t})dt.$$
Hence, from (\ref{jv1}), we have
$$\log f(\sqrt{AB})\leq \int_{0}^{1}\log f(A^{t}B^{1-t})dt.$$
This proved left inequality of (\ref{mr}). \\
On the other hand, we have $f(A^{t}B^{1-t})\leq f(A)^{t}f(B)^{1-t}$. It follows that 
\begin{eqnarray*}
\log f(A^{t}B^{1-t})&\leq& \log f(A)^{t}f(B)^{1-t}\\
&=& \log f(A)^{t}+\log f(B)^{1-t}\\
&=&t \log f(A)+(1-t)\log f(B).
\end{eqnarray*}
So, 
\begin{equation}\label{iv}
\log f(A^{t}B^{1-t})\leq t \log f(A)+(1-t)\log f(B).
\end{equation}
Now, integrate of (\ref{iv}) on $[0,1]$, we have
\begin{eqnarray*}
\int_{0}^{1}\log f(A^{t}B^{1-t})dt &\leq& \int_{0}^{1} t \log f(A)dt+\int_{0}^{1}(1-t)\log f(B)dt\\
&=&\log f(A)\int_{0}^{1}tdt+\log f(B)\int_{0}^{1}(1-t) dt\\
&=&\frac{1}{2}\left(\log f(A)+\log f(B)\right)\\
&=&\log \sqrt{f(A)f(B)}.
\end{eqnarray*}
This completes the proof.
\end{proof}

We should mention, when $f$ is operator geometrically convex function, then we have
\begin{eqnarray*}
f(\sqrt{AB})&=&f(\sqrt{A^{t}B^{1-t}A^{1-t}B^{t}})\\
&\leq&\sqrt{f(A^{t}B^{1-t})f(A^{1-t}B^{t})}\\
&\leq&\sqrt{f(A)^{t}f(B)^{1-t}f(A)^{1-t}f(B)^{t}}\\
&=&\sqrt{f(A)f(B)}.
\end{eqnarray*}
So, we have
$$f(\sqrt{AB})\leq\sqrt{f(A^{t}B^{1-t})f(A^{1-t}B^{t})}\leq\sqrt{f(A)f(B)}.$$
Integrate above inequality over $[0,1]$, we obtain

$$f(\sqrt{AB})\leq\int_{0}^{1}\sqrt{f(A^{t}B^{1-t})f(A^{1-t}B^{t})} dt\leq\sqrt{f(A)f(B)},$$
for $0\leq t\leq 1$ and $A,B\in \mathcal{A}^{+}$ such that $\Sp (A), \Sp(B)\subseteq I$.\\

Let $A, B\in \mathcal{A}$ and $A\leq B$, by continuous functional calculus \cite[Theorem 10.3 (b)]{zhu}, we can easily obtain $\exp(A)\leq \exp(B)$. This means $\exp(t)$ is operator monotone on $[0,\infty)$ for $A,B\in\mathcal{A}$.

 On the other hand, like the classical case, the arithmetic-geometric mean inequality
holds for operators as following
\begin{equation}\label{ag}
A^{\frac{1}{2}}\left(A^{-\frac{1}{2}}BA^{-\frac{1}{2}}\right)^{\nu}A^{\frac{1}{2}}\leq (1-\nu) A+\nu B, \ \ \nu\in[0,1]
\end{equation}
with respect to operator order for positive non-commutative operator in $B(H)$. Whenever, $A$ and $B$ commute together, then inequality (\ref{ag}) reduces to 
\begin{equation}
A^{1-\nu}B^{\nu}\leq (1-\nu)A+\nu B, \ \ \nu\in[0,1].
\end{equation}
Since $\exp(t)$ is an operator monotone function, by above inequality we have 
\begin{eqnarray*}
\exp\left({A^{1-\nu}B^{\nu}}\right)&\leq & \exp\left((1-\nu)A+\nu B\right)\\
&=& \exp((1-\nu)A)\exp(\nu B)\\
&=& \exp(A)^{1-\nu}\exp(B)^{\nu},
\end{eqnarray*}
for $A, B\in \mathcal{A}^{+}$ and $\nu\in [0,1]$. So, in this case $\exp(t)$ is an operator geometrically convex function on $[0,\infty)$.

Let replace $f$ in Theorem \ref{thm} by $\exp(t)$ as an operator geometrically convex function, we have

\begin{eqnarray*}
\log \exp(\sqrt{AB})\leq \int_{0}^{1}\log \exp(A^{t}B^{1-t})dt&\leq& \log \sqrt{\exp(A)\exp(B)}\\
&=&\frac{1}{2}\log \left(\exp(A)\exp(B)\right)\\
&=&\frac{1}{2}\left(\log \exp(A)+\log \exp(B)\right).
\end{eqnarray*}
So, 
\begin{equation}\label{mv1}
\sqrt{AB}\leq \int_{0}^{1}A^{t}B^{1-t}dt\leq\frac{A+B}{2},
\end{equation}
for $A, B\in \mathcal{A}^{+}$.

Here, we mention some remarks for operator geometrically convex functions.
\begin{remark}
$f(x)=\|x\|$ is geometrically convex function for usual operator norms since the following hold
 \begin{equation*}
 f(A^{\alpha}B^{1-\alpha})=\|A^{\alpha}B^{1-\alpha} \|\leq \|A\|^{\alpha}\|B\|^{1-\alpha}=f(A)^{\alpha}f(B)^{1-\alpha}.
 \end{equation*}
Above inequality is a special case of McIntosh inequality.

\end{remark}
\begin{remark}
If $f(t)$ is an operator geometrically convex function, then so is $g(t)=tf(t)$ 
\begin{eqnarray*}
g(A^{\alpha}B^{1-\alpha})&=&A^{\alpha}B^{1-\alpha}f(A^{\alpha}B^{1-\alpha})\\
&\leq&A^{\alpha}B^{1-\alpha}f(A)^{\alpha}f(B)^{1-\alpha}\\
&\leq&A^{\alpha}f(A)^{\alpha}B^{1-\alpha}f(B)^{1-\alpha}\\
&=&g(A)^{\alpha}g(B)^{1-\alpha}
\end{eqnarray*}
for $\alpha\in[0,1]$ and $A, B\in \mathcal{A}^{+}$.
\end{remark}

\begin{remark}
Operator geometrically convex functions is an algebra with some complication of operators spectra. To see this we make use of the following inequality
\begin{equation}\label{shak}
A^{\alpha}B^{1-\alpha}+C^{\alpha}D^{1-\alpha}\leq (A+C)^{\alpha}+(B+D)^{1-\alpha}
\end{equation}
for $A, B, C, D\in \mathcal{A}^{+}$.

Let $f$ and $g$ be operator geometrically convex functions.\\
First, we prove that $f+g$ is an operator geometrically convex function
\begin{eqnarray*}
(f+g)(A^{\alpha}B^{1-\alpha})&=& f(A^{\alpha}B^{1-\alpha})+g(A^{\alpha}B^{1-\alpha})\\
&\leq& f(A)^{\alpha}f(B)^{1-\alpha}+g(A)^{\alpha}g(B)^{1-\alpha}\\
&\leq& \left(f(A)+g(A)\right)^{\alpha}+\left(f(B)+g(B)\right)^{1-\alpha}\\
&=& \left((f+g)(A)\right)^{\alpha}+\left((f+g)(B)\right)^{1-\alpha}
\end{eqnarray*}
for $A, B\in \mathcal{A}^{+}$. In the last inequality above we applied (\ref{shak}).

Second, we show that $mf$ is an operator geometrically convex function for a scalar $m$
\begin{eqnarray*}
(mf)(A^{\alpha}B^{1-\alpha})&\leq& mf(A)^{\alpha}f(B)^{1-\alpha}\\
&=&(mf(A))^{\alpha}(mf(B))^{1-\alpha}
\end{eqnarray*}
for $A, B\in \mathcal{A}^{+}$.

Third,  $h=fg$ is an operator geometrically convex function
\begin{eqnarray*}
h(A^{\alpha}B^{1-\alpha})&=&f(A^{\alpha}B^{1-\alpha})g(A^{\alpha}B^{1-\alpha})\\
&\leq& f(A)^{\alpha}f(B)^{1-\alpha}g(A)^{\alpha}g(B)^{1-\alpha}\\
&=& f(A)^{\alpha}g(A)^{\alpha}f(B)^{1-\alpha}g(B)^{1-\alpha}\\
&=&h(A)^{\alpha}h(B)^{1-\alpha}
\end{eqnarray*}
for $A, B\in \mathcal{A}^{+}$.
\end{remark}

Let $\{e_{i}\}_{i\in I}$ be an orthonormal basis of $H$, we say that $A\in B(H)$ is \textit{trace class} if 
\begin{equation}\label{dd7}
\|A\|_{1}:=\sum_{i\in I}\langle |A|e_{i},e_{i}\rangle <\infty.
\end{equation}
The definition of $\|A\|_{1}$ does not depend on the choice of the orthonormal basis $\{e_{i}\}_{i\in I}$. We denote by $B_{1}(H)$ the set of trace class operators in $B(H)$.

We define the \textit{trace} of a trace class operator $A\in B_{1}(H)$ to be 
\begin{equation}\label{dd9}
\Tr(A):=\sum_{i\in I}\langle Ae_{i}, e_{i}\rangle,
\end{equation}
where $\{e_{i}\}_{i\in I}$ an orthonormal basis of $H$. 

Note that this coincides with the usual definition of the trace if $H$ is finite-dimensional. We observe that the series (\ref{dd9}) converges absolutely.\\

The following result collects some properties of the trace:

\begin{theorem}
We have

(i) If $A\in B_{1}(H)$ then $A^{*}\in B_{1}(H)$ and 
\begin{equation}\label{dd10}
\Tr(A^{*})=\overline{\Tr(A)};
\end{equation}

(ii) If $A\in B_{1}(H)$ and $T\in B(H)$, then $AT, TA\in B_{1}(H)$ and
\begin{equation}\label{dd11}
\Tr(AT)=\Tr(TA) \ \ \ and \ \ \ |\Tr(AT)|\leq \|A\|_{1}\|T\|;
\end{equation}

(iii) $\Tr(\cdot)$ is a bounded linear functional on $B_{1}(H)$ with $\|\Tr\|=1$;

(iv) If $A, B\in B_{1}(H)$ then  $\Tr(AB)=\Tr(BA)$.

\end{theorem}

For the theory of trace functionals and their applications the reader is referred to \cite{sim}.\\

For $A,B\geq 0$ we have $\Tr (AB)\leq \Tr(A)\Tr(B)$. Also, since  $f(t)=t^{\frac{1}{2}}$ is monotone  we have
\begin{equation}\label{l1}
\sqrt{\Tr(AB)}\leq \sqrt{\Tr(A)\Tr(B)}
\end{equation}
for positive operator $A$ and $B$ in $B(H)$.

We know that $f(t)=\Tr(t)$ is operator geometrically convex function \cite[p.513]{hor}, i.e. 
$$\Tr(A^{t}B^{1-t})\leq \Tr (A)^{t}\Tr(B)^{1-t}$$
for $0\leq t\leq 1$ and positive operators $A,B\in B_{1}(H) $.\\
For commutative case, we have
$$\sqrt{\Tr(AB)}\leq \Tr(\sqrt{AB})\leq \sqrt{\Tr(A)\Tr(B)},$$
since $(\Tr (AB))^{\frac{1}{2}}\leq \Tr(AB)^{\frac{1}{2}}$.\\

\noindent Moreover, by Theorem \ref{thm} we can write
\begin{eqnarray*}
\log \Tr(\sqrt{AB})&\leq& \int_{0}^{1}\log\Tr(A^{t}B^{1-t})dt\\
&\leq&\log\sqrt{\Tr (A)\Tr (B)}\\
&=&\frac{1}{2}(\log \Tr(A)+\log \Tr(B)).
\end{eqnarray*}

Let replace $A$ and $B$ by $A^{2}$ and $B^{2}$ in above inequality, respectively. By applying commutativity of algebra and knowing that $\Tr(A)^{2}\leq (\Tr A)^{2}$ for positive operator $A$, we have
\begin{equation*}
\log \Tr(AB)\leq \int_{0}^{1}\log\Tr(A^{2t}B^{2(1-t)})dt\leq\log\left(\Tr (A)\Tr (B)\right).
\end{equation*}

\section{\textbf{More results on trace functional class for product of operators}}
In this section we prove some trace functional class inequalities for operators which are not necessarily commutative.

We consider the wide class of unitarily invariant
norms $|||\cdot|||$. Each of these norms is defined on an ideal in
$B(H)$ and it will be implicitly understood that when we talk of
$|||T|||$, then the operator $T$ belongs to the norm ideal
associated with $|||\cdot|||$. Each unitarily invariant norm
$|||\cdot|||$ is characterized by the invariance property
$|||UTV|||=|||T|||$ for all operators $T$ in the norm ideal
associated with $|||\cdot|||$ and for all unitary operators $U$ and
$V$ in $B(H)$.
 For $1\leq
p<\infty$, the Schatten $p$-norm of an operator $A\in B_{1}(H)$
defined by $\|A\|_{p}=(\Tr |A|^{p})^{1/p}$. These Schatten
$p$-norms are unitarily invariant.

In \cite{bha2}, Bhatia and Davis proved the following inequality
\begin{equation}\label{bha}
||||A^{*}XB|^{r}|||^{2}\leq ||||AA^{*}X|^{r}|||.||||XBB^{*}|^{r}|||
\end{equation}
for all operators $A$, $B$, $X$ and $r\geq 0$.\\
As we know,  $\|A\|_{1}=\Tr|A|$. 
From (\ref{bha}) for $p=1$, we have 
\begin{equation}\label{e1}
\||A^{*}XB|^{r}\|_{1}^{2}\leq\||AA^{*}X|^{r}\|_{1}.\||XBB^{*}|^{r}\|_{1}.
\end{equation}
So, by inequality (\ref{e1}), we can write  
\begin{equation}\label{ph}
(\Tr|A^{*}XB|^{r})^{2}\leq \Tr(|AA^{*}X|^{r})\Tr(|XBB^{*}|^{r}),
\end{equation}
for all operators $A, B\in B_{1}(H)$, $X\in B(H)$ and $r\geq 0$.\\
Let, $X=I$ in above inequality, we have
\begin{equation*}
(\Tr|A^{*}B|^{r})^{2}\leq \Tr(|AA^{*}|^{r})\Tr(|BB^{*}|^{r}).
\end{equation*}
Moreover, let $r=1$ in inequality (\ref{ph}), we have
\begin{equation*}
|\Tr(A^{*}XB)|^{2}\leq\left(\Tr|A^{*}XB|\right)^{2}\leq \Tr(|AA^{*}X|)\Tr(|XBB^{*}|).
\end{equation*}
Put $X^{*}$ instead of $X$  and applying the property of trace we have
\begin{equation}\label{thm23}
|\Tr(AB^{*}X)|^{2}\leq \Tr(|AA^{*}X^{*}|)\Tr(|X^{*}BB^{*}|),
\end{equation}
for all $A, B\in B_{1}(H)$ and $X\in B(H)$.\\

Let $X=I$ in (\ref{thm23}).
\begin{corollary}\label{c1}
Let $A, B\in B_{1}(H)$. Then 
\begin{equation}\label{and}
|\Tr(AB^{*})|^{2}\leq \Tr(AA^{*})\Tr(BB^{*}).
\end{equation}

\end{corollary}

In \cite[Theorem 5]{dra}, Dragomir proved the following inequality for $X\in B(H)$, $A,B\in B_{1}(H)$ and $\alpha\in[0,1]$

$$|\Tr(AB^{*}X)|^{2}\leq \Tr\left(|A^{*}|^{2}|X|^{2\alpha}\right)\Tr\left(|B^{*}|^{2}|X^{*}|^{2(1-\alpha)}\right).$$

Here, we give a generalization for above inequality when $\alpha\in\mathbb{R}$.
\begin{theorem}\label{pos}
Let $X\in B_{1}(H)$, $A, B\in B(H)$ and $\alpha\in \mathbb{R}$. Then
\begin{equation}\label{rdra}
|\Tr(AB^{*}|X|)|^{2}\leq \Tr\left(|A^{*}|^{2}|X|^{2\alpha}\right)\Tr\left(|B^{*}|^{2}|X|^{2(1-\alpha)}\right).
\end{equation}
\end{theorem}
\begin{proof}
Let replace $A$ and $B$ in Corollary \ref{c1} with $|X|^{\alpha}A$ and $|X|^{(1-\alpha)}B$, where $\alpha\in \mathbb{R}$. It follows that
\begin{eqnarray*}
|\Tr(AB^{*}|X|)|^{2}&\leq& \Tr\left(|X|^{\alpha}AA^{*}|X|^{\alpha}\right)\Tr\left(|X|^{(1-\alpha)}BB^{*}|X|^{(1-\alpha)}\right)\\
&=&\Tr\left(AA^{*}|X|^{\alpha}|X|^{\alpha}\right)\Tr\left(BB^{*}|X|^{(1-\alpha)}|X|^{(1-\alpha)}\right)\\
&=&\Tr\left(|A^{*}|^{2}|X|^{2\alpha}\right)\Tr\left(|B^{*}|^{2}|X|^{2(1-\alpha)}\right).
\end{eqnarray*}
So, we have 
$$|\Tr(AB^{*}|X|)|^{2}\leq \Tr\left(|A^{*}|^{2}|X|^{2\alpha}\right)\Tr\left(|B^{*}|^{2}|X|^{2(1-\alpha)}\right).$$
\end{proof}
\noindent Let $A=B=I$ in Theorem \ref{pos}, we have
$$|\Tr(X)|^{2}\leq \Tr\left(|X|^{2\alpha}\right)\Tr\left(|X|^{2(1-\alpha)}\right),$$
for $X\in B_{1}(H)$ and $\alpha\in\mathbb{R}$. Above inequality is a refinement for \cite[Inequality (3.1)]{dra}.\\
Also,
 let $X\in B_{1}(H)$ and normal operators $A, B\in B(H)$. For $\alpha\in \mathbb{R}$, we have
\begin{equation*}
|\Tr(AB^{*}|X|)|^{2}\leq \Tr\left(|A|^{2}|X|^{2\alpha}\right)\Tr\left(|B|^{2}|X|^{2(1-\alpha)}\right).
\end{equation*}

In \cite[Theorem 2.3]{dan}, F. M. Dannan proved that if $S_{i}$ and $T_{i}$ ($i=1,2,\ldots,n$) are positive definite matrices, then we have
\begin{equation}\label{pd}
\left(\Tr\sum_{i=1}^{n}S_{i}T_{i}\right)^{2}\leq \Tr\left(\sum_{i=1}^{n}S_{i}^{2}\right)\Tr\left(\sum_{i=1}^{n}T_{i}^{2}\right).
\end{equation}

\noindent Moreover, if
$S_{i}T_{i}\geq 0$, ($i=1,2,\ldots,n$). Then
\begin{eqnarray*}
\Tr\left(\sum_{i=1}^{n}S_{i}T_{i}\right)^{2}&\leq& \left(\Tr\sum_{i=1}^{n}S_{i}T_{i}\right)^{2}\\
&\leq&\Tr\left(\sum_{i=1}^{n}S_{i}^{2}\right)\Tr\left(\sum_{i=1}^{n}T_{i}^{2}\right).
\end{eqnarray*}
So,

$$\left(\Tr\sum_{i=1}^{n}S_{i}T_{i}\right)^{2}\leq \Tr\left(\sum_{i=1}^{n}S_{i}^{2}\right)\Tr\left(\sum_{i=1}^{n}T_{i}^{2}\right).$$

Here, we prove inequality (\ref{pd}) for arbitrary operators.

\begin{theorem}\label{poco}
Let $S_{i}$ and $T_{i}$ ($i=1,2,\ldots,n$) be arbitrary operators in $B_{1}(H)$. Then,
\begin{equation}
\left|\Tr\left(\sum_{i=1}^{n}S_{i}T_{i}^{*}\right)\right|^{2}\leq \Tr\left(\sum_{i=1}^{n}S_{i}S_{i}^{*}\right)\Tr\left(\sum_{i=1}^{n}T_{i}T_{i}^{*}\right).
\end{equation}
\end{theorem}
\begin{proof}
Let $A=\left(
    \begin{array}{ccccc}
    S_{1} & S_{2}&\ldots  &S_{n}  \\
    0&  0 & \ldots   &0  \\
     \vdots  &\vdots & \ddots   &\vdots  \\
        0  &  0       &   0     & 0\\
  \end{array}\right)$ and $B=\left(
    \begin{array}{ccccc}
    T_{1} & T_{2}&\ldots  &T_{n}  \\
    0&  0 & \ldots   &0  \\
     \vdots  &\vdots & \ddots   &\vdots  \\
        0  &  0       &   0     & 0\\
  \end{array}\right)$. So, we have
 $$AB^{*}=\left(
    \begin{array}{cccc}
   \sum_{i=1}^{n}S_{i}T_{i}^{*} &0&\ldots  &0  \\
    0&  0 & \ldots   &0  \\
     \vdots  &\vdots & \ddots   &\vdots  \\
        0  &  0       &   0     & 0\\
  \end{array}\right),$$ 
  $$AA^{*}=\left(\begin{array}{cccc}
   \sum_{i=1}^{n} S_{i}S_{i}^{*}    &0&\ldots  &0  \\
    0&  0 & \ldots   &0  \\
     \vdots  &\vdots & \ddots   &\vdots  \\
        0  &  0       &   0     & 0\\
  \end{array}\right),$$ 
  $$BB^{*}=\left(\begin{array}{cccc}
   \sum_{i=1}^{n} T_{i}T_{i}^{*}    &0&\ldots  &0  \\
    0&  0 & \ldots   &0  \\
     \vdots  &\vdots & \ddots   &\vdots  \\
        0  &  0       &   0     & 0\\
  \end{array}\right).$$
Put $A$ and $B$ in inequality (\ref{and}), by property of trace, we obtain the desired result.

\end{proof}

\begin{corollary}
Let $S_{i}$ and $T_{i}$ ($i=1,2,\ldots,n$) be positive operators in $B_{1}(H)$. Then,  we have 
$$\left(\Tr\sum_{i=1}^{n}S_{i}T_{i}\right)^{2}\leq \Tr\left(\sum_{i=1}^{n}S_{i}^{2}\right)\Tr\left(\sum_{i=1}^{n}T_{i}^{2}\right).$$
\end{corollary}
\begin{proof}
By Theorem \ref{poco} for positive operators $S_{i}$ and $T_{i}$, we obtain
$$\left|\Tr\left(\sum_{i=1}^{n}S_{i}T_{i}\right)\right|^{2}\leq \Tr\left(\sum_{i=1}^{n}S_{i}^{2}\right)\Tr\left(\sum_{i=1}^{n}T_{i}^{2}\right).$$
Since $S_{i}$ and $T_{i}$ are positive operators, we have $\Tr(S_{i}T_{i})\geq0$. It follows that $\Tr(\sum_{i=1}^{n}S_{i}T_{i})\geq 0$ because $\Tr\left(\sum_{i=1}^{n}S_{i}T_{i}\right)=\sum_{i=1}^{n}\Tr(S_{i}T_{i})$. So, 
$$\left(\Tr\sum_{i=1}^{n}S_{i}T_{i}\right)^{2}\leq \Tr\left(\sum_{i=1}^{n}S_{i}^{2}\right)\Tr\left(\sum_{i=1}^{n}T_{i}^{2}\right).$$
\end{proof}


\end{document}